\documentclass{amsart}
\usepackage{amsmath,amssymb,amsthm,url,comment}
\usepackage{graphicx}

\newtheorem{theorem}{Theorem}[section]  

\newtheorem{proposition}[theorem]{Proposition} 
\newtheorem{conjecture}{Conjecture} 
 
\newtheorem{corollary}[theorem]{Corollary}
\theoremstyle{definition}

\newtheorem{remark}[theorem]{Remark}
\newtheorem*{remark*}{Remark}
\newcommand{\Z}{\mathbb{Z}}
\newcommand{\R}{\mathbb{R}}

\newcommand{\Q}{\mathbb{Q}}
\newcommand{\strutxx}{
\begin{picture}(18,12)
\qbezier(4,4)(10,12)(16,4) 
\put(2,-2){\scriptsize $x$}
\put(14,-2){\scriptsize $x$}
\end{picture} 
}
\newcommand{\strutxy}{
\begin{picture}(18,12)
\qbezier(4,4)(10,12)(16,4) 
\put(2,-2){\scriptsize $x$}
\put(14,-2){\scriptsize $y$}
\end{picture} 
}
\newcommand{\strutyy}{
\begin{picture}(18,12)
\qbezier(4,4)(10,12)(16,4) 
\put(2,-2){\scriptsize $y$}
\put(14,-2){\scriptsize $y$}
\end{picture} 
}

\newcommand{\Atwoxx}{
\raisebox{-2mm}{
\begin{picture}(14,14)
\put(7,4){\oval(12,8)}
\put(3,8){\line(0,1){4}}
\put(10,8){\line(0,1){4}}
\put(0,14){\scriptsize $x$}
\put(10,14){\scriptsize $x$}
\end{picture}
}
} 

\newcommand{\Athreexx}{
\raisebox{-2mm}{
\begin{picture}(14,14)
\put(7,4){\oval(12,8)}
\put(1,4){\line(1,0){12}}
\put(3,8){\line(0,1){4}}
\put(10,8){\line(0,1){4}}
\put(0,14){\scriptsize $x$}
\put(10,14){\scriptsize $x$}
\end{picture}
}
} 

\newcommand{\Athreeyy}{
\raisebox{-2mm}{
\begin{picture}(14,14)
\put(7,4){\oval(12,8)}
\put(1,4){\line(1,0){12}}
\put(3,8){\line(0,1){4}}
\put(10,8){\line(0,1){4}}
\put(0,14){\scriptsize $y$}
\put(10,14){\scriptsize $y$}
\end{picture}
}
} 

\newcommand{\Atwoyy}{
\raisebox{-2mm}{
\begin{picture}(14,14)
\put(7,4){\oval(12,8)}
\put(3,8){\line(0,1){4}}
\put(10,8){\line(0,1){4}}
\put(0,14){\scriptsize $y$}
\put(10,14){\scriptsize $y$}
\end{picture} } 
}
\newcommand{\Htree}{
\raisebox{-2mm}{
\begin{picture}(26,14)
\put(6,8){\oval(8,12)[r]}
\put(10,8){\line(1,0){6}}
\put(20,8){\oval(8,12)[l]}
\put(0,12){\scriptsize $x$}
\put(0,0){\scriptsize $y$}
\put(22,12){\scriptsize $x$}
\put(22,0){\scriptsize $y$}
\end{picture} }
}

 \makeatletter
    
    \@addtoreset{equation}{section}
  \makeatother
 
\begin{document}

\title[]{Applications of the Casson-Walker invariant to the knot complement and the cosmetic crossing conjectures}
\author{Tetsuya Ito}


\begin{abstract}
We give a rational surgery formula for the Casson-Walker invariant of a 2-component link in $S^{3}$ which is a generalization of Matveev-Polyak's formula. As application, we give more examples of non-hyperbolic L-space $M$ such that knots in $M$ are determined by their complements. We also apply the result for the cosmetic crossing conjecture.
\end{abstract}

\maketitle

\section{Introduction}

Knots $K$ and $K'$ in an oriented closed 3-manifold $M$ are \emph{equivalent} if there exists an orientation-preserving homeomorphism $f:M\rightarrow M$ such that $f(K)=K'$. When two knots are equivalent, obviously there exists an orientation-preserving homeomorphism $f:M\setminus K \rightarrow M\setminus K'$ between their complements $M\setminus K$ and $M\setminus K'$. In the following, we denote by $M \cong M'$ if two oriented 3-manifolds $M$ and $M'$ are orientation-preservingly homeomorphic.

We say that a knot $K$ is \emph{determined by its complement} if the converse is true; for a knot $K'$ in $M$, if $M \setminus K' \cong M \setminus K$ then $K$ and $K'$ are equivalent.

One can understand a knot determined by its complement in terms of Dehn surgery on $K$. For a slope $r$ of the knot exterior $X(K):= M \setminus N(K)$, where $N(K)$ denotes the tubular neighborhood of $K$, we denote by $M_K(r)$ the $r$-surgery on $K$. We say that two slopes $r$ and $r'$ are \emph{equivalent} if there is an orientation-preserving homeomorphism of $M_K$ that sends $r$ to $r'$. 

Let $K' \subset M_K(r)$ be the core of the attached solid torus $M_K(r)=X(K) \cup (S^{1} \times D^{2})$. Then the complements $M\setminus K$ and $M_K(r) \setminus K'$ are orientation-preservingly homeomorphic. Assume that $M_K(r) \cong M$ so there exists an orientation-preserving homeomorphism $f: M \rightarrow M_K(r)$. Since $f(K)=K'$ if and only if $f$ sends the meridian $\mu_K$ to $r$, a knot $K$ in $M$ is determined by its complement if and only if $M_K(r) \cong M$ implies $r$ is equivalent to the meridian.

A famous Gordon-Luecke theorem \cite{gl} states that knots in $S^{3}$ are determined by their complements. It is conjectured that the same conclusion holds for knots in general $3$-manifolds \cite[Problem 1.81D]{ki}. 

\begin{conjecture} (The oriented knot complement conjecture) 
Knots in a 3-manifold $M$ are determined by their complements; 
For a knot $K$ in a 3-manifold $M$, if $M_K(s) \cong M$, then the surgery slope $s$ is equivalent to the meridian $\mu_K$ of $K$.
\end{conjecture}

By \cite{ma, ro}, non-hyperbolic knots in an irreducible, small Seifert fibered space are also determined by their complements.

\begin{theorem}\cite[Theorem 1.3]{ma}
\label{theorem:Non-hyperbolic}
Non-hyperbolic knots in an irreducible small Seifert fibered space $M$ are determined by their complements.
\end{theorem}

For null-homologous knots, using Heegaard Floer homology Gainullin showed the following.

\begin{theorem}\cite[Theorem 8.2]{ga}
\label{theorem:L-space}
Null-homologous knots in an L-space $M$ are determined by their complements.
\end{theorem}

As was implicit in the argument of \cite[Corollary 8.3]{ga}, for a non-null-homologous knots, the order of the first homology gives the following constraint.

\begin{theorem}
\label{theorem:1sthomology}
Let $K$ be a non-null-homologous knot in a rational homology sphere $M$.
Assume that $M$ is obtained by a Dehn surgery on a knot in $S^{3}$.
If $|H_1(M;\Z)|$ is square-free, for every slope $s$ with $\Delta(s,\mu_K)=1$, $M \not \cong M_K(s)$.
\end{theorem}

Here $\Delta(s,\mu_K)$ denotes the minimum geometric intersection of two slopes $s$ and $\mu_K$. This observation, the cyclic surgery theorem \cite{cgls} and Theorem \ref{theorem:Non-hyperbolic} shows that non-null-homologous knots in $L(p,q)$ are also determined by their complements whenever $p$ is square-free. Consequently,

\begin{theorem}\cite[Corollary 8.3]{ga}
\label{theorem:lens}
If $p$ is square-free then all knots in $L(p,q)$ are determined by their complements. 
\end{theorem}

Ichihara and Saito proved the same result for $L(4,q)$ \cite{is}, by a closer look at the 1st homology group, namely, by checking whether the 1st homology group is cyclic or not. 

The aim of this paper is to extend Theorem \ref{theorem:1sthomology} by using the Casson-Walker invariant $\lambda_{w}$ to generalize Theorem \ref{theorem:lens}.

Recall that for a rational homology sphere, the order of the 1st homology can be seen as a finite type invariant of degree zero, and the Casson-Walker invariant is the finite type invariant of degree one.
Thus our argument can be seen as the simplest generalization of the 1st homology argument in a point of view of the finite type invariants. It is interesting to investigate more constraint coming from higher order (degree $2,3,\ldots$) finite type invariants, as the author did in \cite{ito1} for surgery on knots in $S^{3}$.

To this purpose, in Section \ref{section:Q} we give a \emph{rational} surgery formula of the Casson-Walker invariant of a 2-component link in $S^{3}$. This is a generalization for Matveev-Polyak's formula \cite{mp} of \emph{integral} surgery on a 2-component link in $S^{3}$.

\begin{theorem}[Rational surgery formula of the Casson-Walker invariant  for 2-component links in $S^{3}$]
\label{theorem:Q-MP}
Let $L=K_x \cup K_y$ be a rationally framed 2-component link, where the framings of the components $K_x$ and $K_y$ are $f_x=p_x \slash q_x$ and $f_y=p_y \slash q_y$, respectively. Let $S^{3}_L$ be the 3-manifold obtained by Dehn surgery on $L$.

If $S^{3}_L$ is a rational homology 3-sphere, its Casson-Walker invariant $\lambda_w(S^{3}_L)$ is given by 
\begin{align*}
D\left(\frac{\lambda_w(S^{3}_L)}{2}-\frac{\varsigma}{8}\right) &= a_2(K_x)\frac{p_y}{q_y}-\frac{p_y}{24q_y} -\frac{p_y}{24q_yq_x^2} +\frac{p_y\ell^2}{24q_y}\\
&+a_2(K_y)\frac{p_x}{q_x} -\frac{p_x}{24q_x}-\frac{p_x}{24q_xq_y^2}+\frac{p_x\ell^2}{24q_x}\\
&+ 2v_3(L) + \frac{D}{24}\left( S\left(\frac{p_x}{q_x}\right)-\frac{p_x}{q_x} +S\left(\frac{p_y}{q_y}\right)-\frac{p_y}{q_y} \right).
\end{align*}
Here 
\begin{itemize}
\item $a_{i}(K)$ is the coefficient of $z^{i}$ in the Conway polynomial $\nabla_{K}(z)$ of $K$.
\item $\ell=lk(K_x,K_y)(=a_1(L))$ is the linking number.
\item $D= \det A$ and $\varsigma =\sigma(A)$ be the determinant and the signature of the linking matrix $A=\begin{pmatrix} f_x & \ell \\ \ell & f_y \end{pmatrix}$.
\item $v_3(L)$ is the coefficient of $\Htree$ in the Kontsevich invariant of $L$. This is explicitly written by
\[ v_3(L)=\frac{1}{2}\left( -a_3(L)+(a_2(K_x)+a_2(K_y))\ell +\frac{1}{12}(\ell^3-\ell) \right) \]
\item $S(\frac{p}{q})$ is the Dedekind symbol (see \cite{km} for definition and backgrounds) related to the more famous Dedekind sum $s(p,q)$ by 
\[ S\left(\frac{p}{q}\right)=  12(sign(q))s(p,q). \]
The Dedekind sum $s(p,q)$ is defined by
\[ s(p,q)=\sum_{k=1}^{|q|-1} \left(\!\!\left( \frac{k}{p}\right)\!\!\right)\left(\!\!\left( \frac{kq}{p}\right)\!\!\right) 
\]
where $(\!( x)\!)=x-\lfloor x\rfloor -\frac{1}{2}$.
\end{itemize}
\end{theorem}

The Casson-Walker invariant allows us to extend Theorem \ref{theorem:1sthomology} for the case the order of the 1st homology contains square factor $2^{2},3^2,6^2$.

\begin{theorem}
\label{theorem:main}
Let $M$ be a rational homology sphere which is obtained by the $\frac{p}{q}$-surgery on a knot in $S^{3}$ ($\frac{p}{q}>0$) and let $K$ be a non-null-homologous knot in $M$.
Assume that $p=d^2p'$ such that $p'$ is square-free, $d\in \{1,2,3,6\}$ and that $\gcd(d,p')=1$ when $d>1$.
Then for every slope $s$ with $\Delta(s,\mu_K)=1$, $M_K(s) \not \cong M$.
\end{theorem}

Actually, the Casson-Walker invariant provides more general constraints that can be applied for slope $s$ with $\Delta(s,\mu_K)>1$. As an application, we extend Theorem \ref{theorem:lens} for several non-hyperbolic L-spaces.

\begin{theorem}
\label{theorem:main-finite}
Let $M$ be a 3-manifold obtained from a Dehn surgery on a knot in $S^{3}$.
Assume that $M$ satisfies one of the followings.
\begin{enumerate}
\item[(i)] $M$ is a reducible L-space, such that $|H_1(M;\Z)|=d^2p'$ where $p'$ is square-free, $d\in \{1,2,3,6\}$. Moreover, when $d>1$, $\gcd(d,p')=1$.
\item[(ii)] $M$ is a lens space such that $|H_1(M;\Z)|=d^2p'$ where $p'$ is square-free, $d\in \{1,2,3,6\}$. Moreover, when $d>1$, $\gcd(d,p')=1$.
\item[(iii)] $\pi_1(M)$ is finite and that $|H_1(M;\Z)|$ is square-free.
\item[(iv)]  $M$ is a small Seifert fibered L-space such that $|H_1(M;\Z)|$ is square-free, coprime to both $5$ and $7$, and that $6 \nmid |H_1(M;\Z)|$.
\end{enumerate}
Then all knots in $M$ are determined by their complements. 
\end{theorem}

We also give an application to another famous conjecture \cite[Problem 1.58]{ki}.

\setcounter{conjecture}{1}
\begin{conjecture}[Cosmetic crossing conjecture]
Let $D$ be a diagram of an oriented knot $K$ in $S^{3}$. If the crossing change of $K$ at a crossing $c$ of $D$ preserves $K$ as oriented knot, then $c$ is nugatory.
\end{conjecture}

Here a crossing $c$ is \emph{nugatory} if there is a circle $C$ in the projection plane that transverse to the diagram $D$ with only at the crossing $c$ (see Figure \ref{fig:nugatory} (i) -- for a diagram-free description of nugatory crossing we refer to \cite{ka,lm}. 

\begin{figure}[htbp]
\includegraphics*[width=90mm]{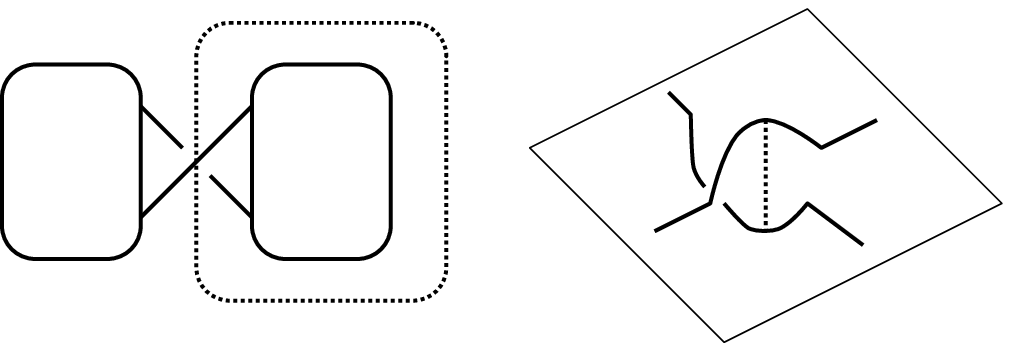}
\begin{picture}(0,0)
\put(-280,80) {(i)}
\put(-130,80) {(ii)}
\put(-220,10) {$C$}
\put(-250,45) {\LARGE $D'$}
\put(-185,45) {\LARGE $D''$}
\put(-80,15) {\rotatebox{-50}{$\R^2$}}
\put(-60,40) {$\gamma$}
\end{picture}
\caption{(i) Nugatory crossing (ii) crossing arc $\gamma$} 
\label{fig:nugatory}
\end{figure} 

The cosmetic crossing conjecture has been verified for several cases; 2-bridge knots \cite{to}, fibered knots \cite{ka}, some satellite knots \cite{bk}, and genus one knots under some assumptions, like non-algebraically-sliceness \cite{bfkp}. Recently extending results in \cite{bfkp}, the author showed the cosmetic crossing conjecture for genus one knots with non-trivial Alexander polynomial \cite{ito2}.

In \cite{lm}, the cosmetic crossing conjecture is confirmed for the following case.

\begin{theorem}\cite{lm}
\label{theorem:Lidman-Moore}
Let $K$ be a knot in $S^{3}$. If the double branched covering $\Sigma(K)$ is an L-space such that each summand of its first homology $H_1(\Sigma(K);\Z)$ has square-free order, then $K$ satisfies the cosmetic crossing conjecture.
\end{theorem}

We give an extension of Theorem \ref{theorem:Lidman-Moore} that allows the case $\det(K)=|H_1(\Sigma(K);\Z)|$ has a square factor $3^{2}$.

\begin{theorem}
\label{theorem:cosmetic}
Let $K$ be a knot in $S^{3}$. Assume that $K$ satisfies the following conditions.
\begin{itemize}
\item[(a)] The double branched covering $\Sigma(K)$ is an L-space.
\item[(b)] $\Sigma(K)$ is obtained by a Dehn surgery on a knot in $S^{3}$.
\item[(c)] $\det(K)=9p'$ where $p'$ is square-free and coprime to $3$.
\end{itemize}
Then $K$ satisfies the cosmetic crossing conjecture.
\end{theorem}

Although the assumption (b) is in general hard to check, by Montesinos trick \cite{mo}, we replace the assumption (b) with a weak but often easier to confirm assumptions.

\begin{corollary}
\label{cor:cosmetic}
Let $K$ be a knot in $S^{3}$. Assume that $K$ satisfies the following conditions.
\begin{itemize}
\item[(a)] The double branched covering $\Sigma(K)$ is an L-space.
\item[(b$'$)] The unknotting number or the $H(2)$-unknotting number of $K$ is one. 
\item[(c)] $\det(K)=9p'$ where $p'$ is square-free and coprime to $3$.
\end{itemize}
Then $K$ satisfies the cosmetic crossing conjecture.
\end{corollary}
Here the $H(2)$-unknotting number of $K$ is the minimum number of $H(2)$-move (see Figure \ref{fig:H2move}) needed to transform $K$ into the unknot (see \cite{kami,nak,zek}).

\begin{figure}[htbp]
\includegraphics*[width=65mm]{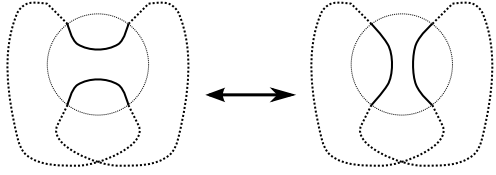}
\caption{$H(2)$-move} 
\label{fig:H2move}
\end{figure} 

According to \cite[Theorem 1.4]{lm}, from the aforementioned known results, the cosmetic crossing conjecture is confirmed for knots with at most ten crossings, except
\[ 10_{65},10_{66},10_{67},10_{77},10_{87},10_{98},10_{108},10_{129},
10_{147},10_{164}.\]

Theorem \ref{theorem:cosmetic} proves the conjecture for half of them,   $10_{65},10_{67},10_{77},10_{108},10_{164}$. 
\begin{corollary}
\label{cor:ten}
Let $K$ be a knot in $S^{3}$ with at most ten crossings. Then $K$ satisfies the cosmetic crossing conjecture, possibly except
\[ 10_{66}, 10_{87}, 10_{98},10_{147}, 10_{129}\]
\end{corollary}

\section*{Acknowledgement}
The author has been partially supported by JSPS KAKENHI Grant Number 19K03490,16H02145. 
He would like to thank K. Ichihara for stimulating conversations and comments for knots detrmined by their complements, and T. Kanenobu for a helpful correspondence.

\section{Rational surgery formula of Casson-Walker invariant}
\label{section:Q}

In this section we prove Theorem \ref{theorem:Q-MP}. Our proof is based on a rational surgery formula of the LMO invariant $Z^{LMO}$ due to Bar-Natan and Lawrence \cite{bl}. Since $\frac{\lambda_w(M)}{4}=Z^{LMO}_1(M)$ \cite{lmo,lmmo}, where $Z^{LMO}_1(M)$ is the degree one part of the LMO invariant, by explicitly writing down the surgery formula of the LMO invariant up to degree one we get the desired surgery formula of the Casson-Walker invariant.

In the following, we assume some familiarity with the theory of the Kontsevich invariant and the LMO invariant. A standard reference is \cite{oht}. In the rational surgery formula, the Aahurs integral construction of the LMO invariant \cite{bgrt1,bgrt2,bgrt3} and the Wheeling theorem \cite{blt} play fundamental roles.

The reader may understand this section as a demonstration of how to apply the rational surgery formula.
See also \cite{ito1}, where we did a similar calculation of the rational surgery formula for degree $2$ and $3$ parts, for the Dehn surgery on knots in $S^{3}$. 

\begin{proof}[Proof of Theorem \ref{theorem:Q-MP}]
Let $L=K_x \cup K_y$ be a rationally framed 2-component link where the framing of the component $K_x$ and $K_y$ are $f_x=\frac{p_x}{q_x}$ and $f_y=\frac{p_y}{q_y}$, respectively. We denote by $L_0$ the link $L$ such that both $K_x$ and $K_y$ have zero framings.

Let $Z(L_0) \in \mathcal{A}(\circlearrowleft_x,\circlearrowleft_y)$ be the Kontsevich invariant of $L_0$. By composing the inverse of the Poincar\'e-Birkoff-Witt isomorphism
\[ \sigma: \mathcal{A}(\circlearrowleft_x,\circlearrowleft_y) \rightarrow \mathcal{A}(\circledast_x,\circledast_y), \]
we view the Kontsevich invariant $Z(L_0)$ so that it takes value in the space $\mathcal{A}(\circledast_x,\circledast_y)$. 
The low-degree parts of $\sigma \circ Z^{\sigma}$ that is relevant in the computation of the degree one part of the LMO invariant is 
\begin{align*}
\sigma\circ Z(L_0) &= 1+ \left(-\frac{a_2(K_x)}{2}+\frac{1}{48}\right)\Atwoxx + \left(-\frac{a_2(K_y)}{2}+\frac{1}{48}\right) \Atwoyy + v_3(L)\Htree\\
& \qquad + \ell \strutxy + \frac{\ell^2}{2} \strutxy \strutxy + \mbox{(other terms)}.
\end{align*}
Here the (other terms) represents the rest of the Kontsevich invariant that are irrelevant to the degree one part of the LMO invariant.

The Wheeled Kontsevich invariant $Z^{\sf Wheel}(L_0):=\widehat{\Omega}_x^{-1}\widehat{\Omega}_y^{-1}(\sigma \circ Z^{\sigma}(L_0))$ is 
\begin{align*}
Z^{\sf Wheel}(L_0) & = 1+
\left(-\frac{a_2(K_x)}{2}+\frac{1}{48}-\frac{\ell^2}{48}\right) \Atwoxx + \left(-\frac{a_2(K_y)}{2}+\frac{1}{48}-\frac{\ell^2}{48}\right) \Atwoyy \\
& \qquad + v_3(L)\Htree +\ell \strutxy + \frac{\ell^2}{2} \strutxy \strutxy + \mbox{(other terms)}
\end{align*}
Here $\widehat{\Omega}_x^{-1}: \mathcal{A}(\circledast_x,\circledast_y)  \rightarrow \mathcal{A}(\circledast_x,\circledast_y) $ is a map defined as follows. 
First, $\Omega_x$ is the wheel element, the Kontsevich invariant of the unknot \cite{blt}
\[ \Omega_x = 1+ \frac{1}{48}\Atwoxx+\cdots \] 
whose legs are marked by $x$. For a Jacobi diagram $D$, $\widehat{\Omega}_x^{-1}(D)$ is the sum of all ways of gluing the legs of 
\[ \Omega_x^{-1} = 1 - \frac{1}{48}\Atwoxx+\cdots\]
 with some legs of $D$ marked by $x$. The definition of $\Omega_y$ and $\widehat{\Omega}_y^{-1}$ are similar.

By the rational surgery formula \cite[Theorem 6]{bl} of the LMO invariant, if $S^{3}_L$ is a rational homology sphere, $Z^{LMO}(S^{3}_L)$ is written by the formal Gaussian integration
\begin{align*}
Z^{LMO}(S^{3}_L) & = \exp\left( \frac{\theta}{48} \left(3 \varsigma(L)+ S(f_x)-f_x + S(f_y)-f_y\right) \right) \\
& \quad \cdot \int^{FG} Z^{\sf Wheel}(L_0)\sqcup \exp(\frac{f_x}{2}\strutxx) \sqcup \exp( \frac{f_y}{2} \strutyy) \sqcup \Omega_{x\slash q_x} \sqcup \Omega_{y \slash q_y} dxdy.
\end{align*}
Here 
\begin{itemize}
\item $\sqcup$ denotes the disjoint union product.
\item $\displaystyle\Omega_{x\slash f_x} = 1+ \frac{1}{48q_x^{2}} \Atwoxx + \cdots \quad (\mbox{resp. } \Omega_{y\slash f_y} = 1+ \frac{1}{48q_y^{2}} \Atwoyy + \cdots)$
is the wheel element $\Omega$ whose legs are marked by $\frac{x}{q_x}$ (resp. $\frac{y}{q_y}$).
\item  $\theta$ denotes the theta-shaped Jacobi diagram. 
\item $\varsigma(L)$ is the signature of the linking matrix of $L$.
\end{itemize}
Then
\begin{align*}
& Z^{\sf Wheel}(L_0)\sqcup \exp(\frac{f_x}{2} \strutxx) \sqcup \exp( \frac{f_y}{2} \strutyy) \sqcup \Omega_{x\slash f_x} \sqcup \Omega_{y \slash f_y}\\
=& \Biggl( 1+
\left(-\frac{a_2(K_x)}{2}+\frac{1}{48}-\frac{\ell^2}{48} +\frac{1}{48q_x^2} \right) \Atwoxx \\
& + \left(-\frac{a_2(K_y)}{2}+\frac{1}{48}-\frac{\ell^2}{48} +\frac{1}{48q_y^2}\right) \Atwoyy  + v_3(L)\Htree +\mbox{(other terms)} \Biggr) \\
& \sqcup \exp\left( \frac{f_x}{2} \strutxx + \ell \strutxy + \frac{f_y}{2} \strutyy \right).
\end{align*}
Since the inverse of the linking matrix is $\displaystyle\begin{pmatrix} f_x & \ell \\ \ell & f_y \end{pmatrix}^{-1} = \frac{1}{D}\begin{pmatrix} f_y & -\ell \\ -\ell & f_x \end{pmatrix}$ where $D \neq 0$ is the determinant,  following the definition of the formal Gaussian integration (see \cite{bgrt1,bgrt2,bgrt3}, 
\begin{align*}
&\int^{FG} Z^{\sf Wheel}(L_0)\sqcup \exp(\frac{f_x}{2}\strutxx) \sqcup \exp( \frac{f_y}{2} \strutyy) \sqcup \Omega_{x\slash q_x} \sqcup \Omega_{y \slash q_y} dxdy \\
= & \Biggl\langle 1+ \left(-\frac{a_2(K_x)}{2}+\frac{1}{48}-\frac{\ell^2}{48} + \frac{1}{48q_x^2}\right) \Atwoxx \\
& + \left(-\frac{a_2(K_y)}{2}+\frac{1}{48}-\frac{\ell^2}{48}+\frac{1}{48q_y^2}\right) \Atwoyy + v_3(L)\Htree + \mbox{(other terms)} \\
& ,  \exp\left( -\frac{f_y}{2D} \strutxx + \frac{\ell}{D} \strutxy - \frac{f_x}{2D} \strutyy \right)  \Biggr\rangle 
\end{align*}

Here $\langle X, Y \rangle$ means the sum of all ways of gluing the legs of $X$ to the legs of $Y$ so that their markings coincide. When the number of $x$ and $y$ marked legs of $X$ and $Y$ are not the same, the pairing is defined to be zero.

Let $[\int^{FG}]_1$ be the degree one part of the formal Gaussian integral $\displaystyle \int^{FG} Z^{\sf Wheel}(L_0)\sqcup \exp(\frac{f_x}{2}\strutxx) \sqcup \exp( \frac{f_y}{2} \strutyy) \sqcup \Omega_{x\slash q_x} \sqcup \Omega_{y \slash q_y} dxdy$.
Then 
\begin{align*}
[{\textstyle\int^{FG}}]_1&=  \left \langle\left(-\frac{a_2(K_x)}{2}+\frac{1}{48}-\frac{\ell^2}{48} +\frac{1}{48q_x^2} \right) \Atwoxx, -\frac{f_x}{2D}\strutxx \right \rangle\\ 
& +  \left\langle \left(-\frac{a_2(K_y)}{2}+\frac{1}{48}-\frac{\ell^2}{48}+\frac{1}{48q_y^2}\right) \Atwoyy, -\frac{f_y}{2D}\strutyy \right \rangle\\
&+ \left\langle  v_3(L)\Htree, \frac{f_xf_y}{4D^2} \strutxx \strutyy \right \rangle + \left \langle  v_3(L)\Htree, \frac{\ell^2}{2D^2} \strutxy \strutxy \right \rangle.\\
\end{align*}
Since 
\begin{align*}
&\left \langle \Atwoxx, \strutxx \right \rangle = 2\theta, & &\left\langle\Atwoyy, \strutyy \right \rangle=2\theta\\
&\left\langle \Htree,  \strutxx \strutyy \right \rangle = 4 \theta,& & \left\langle \Htree,  \strutxy \strutxy \right \rangle = -2\theta,\\
\end{align*}
we get
\begin{align*}
[{\textstyle\int^{FG}}]_1 = & \frac{1}{D} \biggl( \frac{a_2(K_x)f_y}{2}-\frac{f_y}{48}+\frac{\ell^2f_y}{48} -\frac{f_y}{48q_x^2}\\
 & \qquad + \frac{a_2(K_y)f_x}{2} - \frac{f_x}{48}+ \frac{\ell^2f_x}{48} -\frac{f_x}{48q_y^2} + v_3(L)\biggr)\theta.
\end{align*}
Therefore we conclude  
\begin{align*}
\frac{\lambda_w(S^{3}_L)}{4} = &\frac{1}{D} \biggl( \frac{a_2(K_x)f_y}{2}-\frac{f_y}{48}+\frac{\ell^2f_y}{48} -\frac{f_y}{48q_x^2} \\
& \qquad + \frac{a_2(K_y)f_x}{2}-\frac{f_x}{48}+\frac{\ell^2f_x}{48} -\frac{f_x}{48q_y^2} + v_3(L)\biggr) \\
&+ \frac{1}{48}\left( 3 \varsigma(L) + S(f_x)-f_x + S(f_y)-f_y\right)
\end{align*}
as desired.
\end{proof}

\begin{remark}
In the above proof we do not give an explicit formula 
\begin{equation}
\label{eqn:v3} v_3(L)=\frac{1}{2}\left(-a_3(L)+(a_2(K_x)+a_2(K_y))\ell +\frac{1}{12}(\ell^3-\ell)\right) 
\end{equation}
of $v_3(L)$ in terms of the coefficients of the Conway polynomials $\nabla_L(z)$, $\nabla_{K_x}(z)$ and $\nabla_{K_y}(z)$.
 
A cheating, but the easiest way to get the formula (\ref{eqn:v3}) is to compare the rational surgery formula with Matveev-Polyak's integer surgery formula. 

A theoretically more satisfactory way to prove (\ref{eqn:v3}) is to use the Alexander-Conway polynomial weight system evaluations (see \cite{bg,fkv,mv} for details of the Alexander-Conway weight system). Compare \cite[Lemma 2.1]{ito1}, where we used the $\mathfrak{sl}_2$ weight system evaluations to obtain an explicit formula of various coefficients of the Kontsevich invariant of knots.

Let $(\sigma\circ Z)_3(L_0)$ be the degree 3 part of the Kontsevich invariant of $L_0$ and let $W$ be the Alexander-Conway polynomial weight system. Then
\[ W((\sigma\circ Z)_3(L_0)) = a_3(L). \] 
On the other hand, $(\sigma\circ Z)_3(L_0)$ is given by 
\begin{align*}
(\sigma\circ Z)_3(L_0) &= v_3(L) \Htree + w_3(K_x) \Athreexx +w_3(K_y) \Athreeyy \\
 &\quad + \left(-\frac{a_2(K_x)}{2}+\frac{1}{48}\right)\ell \Atwoxx \strutxy + \left(-\frac{a_2(K_y)}{2}+\frac{1}{48}\right)\ell \Atwoxx \strutxy \\
&\quad + \frac{\ell^3}{6} \strutxy \strutxy \strutxy
\end{align*}
Here $w_3(K_x)$ and $w_3(K_y)$ are the primitive degree 3 finite type invariant of the knots $K_x$ and $K_y$, respectively. More explicitly, $w_3(K_x) = -\frac{1}{24}j_3(K_x)$, where $j_3(K_x)$ is the coefficient of $h^{3}$ for the Jones polynomial $V_{K_x}(e^{h})$ \cite[Lemma 2.1]{ito1}.

Since
\begin{align*}
& W(\Htree) = -2, \quad W(\Athreexx) = W(\Athreeyy)=0,\\
& W(\Atwoxx \strutxy) =W(\Atwoyy \strutxy)=-2, \quad W(\strutxy \strutxy \strutxy) =\frac{1}{2} \end{align*}
we get the desired formula
\[ a_3(L) = -2v_3(L) + (a_2(K_x)+a_2(K_y))\ell + \frac{1}{12}(\ell^3-\ell).\]
\end{remark}

\section{Constraint for Dehn surgery on non-null-homologous knot}
\label{section:CW}

Let $M$ be a 3-manifold obtained by the $\frac{p}{q}$-surgery on a knot $K_y$ in $S^{3}$ $(p>0)$. We view $M$ as $M=X(K_y)\cup(S^{1}\times D^{2})$, where $X(K_y)=S^{3} \setminus N(K)$ is the complement of the tubular neighborhood of $K_y$. For a knot $K$ in $M$, we put $K$ so that $K \subset X(K_y) \subset M=X(K_y)\cup (S^{1} \times D^{2})$. Let $K_x$ be the knot in $S^{3}$ which is the image of $K$ under the natural embedding $\iota: X(K_y) \hookrightarrow S^{3} = X(K_y) \cup N(K_y)$. We call the knot $K_x$ in $S^{3}$ \emph{a knot corresponding to the knot $K$}.

For a slope $s$ of the $K \subset M$, let $n=\Delta(s,\mu_K)$.
Since the natural embedding $\iota$ sends the meridian $\mu_{K}$ of the knot $K \subset M$ to the meridian $\mu_{K_x}$ of the knot $K_x \subset S^{3}$, $\iota(s) = \frac{m}{n}$ for some $m \in \Z$. Here for a knot in $S^{3}$ we identify the set of slopes with rational number $\Q$ as usual.

Let $L = K_x \cup K_y$ be the (rationally framed) link in $S^{3}$, where the framing of $K_x$ and $K_y$ are $\frac{m}{n}$ and $\frac{p}{q}$, respectively.
Then $M_K(s) = S^{3}_{L}$ so we are able to use the rational surgery formula to compute $\lambda_w(M_K(s))$.

Under this setting, we study a constraint for $M_K(s) =S^{3}_L=M$ to happen.
Let 
\[ A=\begin{pmatrix} \frac{m}{n} & \ell \\ \ell & \frac{p}{q} \end{pmatrix}\]
be the linking matrix of $L$, where $\ell=lk(L_x,L_y)$ is the linking number of $K_x$ and $K_y$. 

Let $c=\gcd(n,p)$, and we put $p=cp_0$ and $n=cn_0$. 
As the first step, we review the constraint from the order of the 1st homology group in a general setting.

\begin{proposition}
\label{proposition:1sthomology}
Assume that $S^{3}_L=M$.
\begin{enumerate}
\item[(i)] $m=\frac{nq\ell^{2}}{p}+ \varepsilon =\frac{n_0q\ell^{2}}{p_0} + \varepsilon$ 
for $\varepsilon \in \{\pm 1\}$.
\item[(ii)] Let $d_0=p_0 \slash \gcd(p_0,\ell)$. Then $p_0=d_0^2 p_0', \ell = d_0 p_0' \ell'$ with $\gcd(d_0,\ell')=1$. 
\end{enumerate}
\end{proposition}
\begin{proof}
By comparing the order of the 1st homology group we get 
\begin{equation}
\label{eqn:1sthomology} |H_1(S^{3}_L;\Z)| = |mp-nq\ell^{2}| = p = |H_1(M;\Z)| 
\end{equation}
Hence $mp-nq\ell^2=\varepsilon p$ for $\varepsilon \in \{\pm 1\}$.

To see (ii), we put $p_0= d_0 \gcd(p_0,\ell)$ and $\ell = \gcd(p_0,\ell)\ell'$. Then clearly $\gcd(d_0,\ell')=1$. By (i) $p_0 \mid \ell^2$ so $d_0 \mid \gcd(p_0,\ell) \ell'^{2}$. Since $\gcd(d_0, \ell')=1$, this means that $d_0 \mid \gcd(p_0,\ell)$ hence one can write $\gcd(p_0,\ell)=d_0p'_0$. Consequently we get $p_0=d_0^{2}p_0'$, $\ell=d_0p_0'\ell'$ as desired. 
\end{proof}

Theorem \ref{theorem:1sthomology} in introduction is a direct consequence of the proposition.
\begin{proof}[Proof of Theorem \ref{theorem:1sthomology}]
Assume that $S^{3}_L \cong M$ and that $n=\Delta(s,\mu_K)=1$. Then $c=\gcd(p,n)=1$ so $p_0=p$. By Proposition \ref{proposition:1sthomology} (ii) $p=d_0^{2}p_0'$, $\ell=d_0p_0'\ell'$ with $d_0=p\slash \gcd(p,\ell)$. However, since $p=|H_1(M;\Z)|$ is square-free, $d_0=1$. This implies $p\mid \ell$, which contradicts the assumption that $K$ is a non-mull-homologous knot in $M$.
\end{proof}

We proceed to use the Casson-Walker invariant to improve Proposition
\ref{proposition:1sthomology}. The Casson-Walker invariant counterpart of the equation  (\ref{eqn:1sthomology}) is given as follows. 

\begin{proposition}
\label{proposition:CW}
Assume that $S^{3}_L=M$. Then
\begin{align*}
0 &= 24a_2(K_x)np+24a_2(K_y)q(m-\varepsilon)+(3\varepsilon \varsigma-3 \varepsilon -n)p+(np+mq)\ell^{2}\\
\nonumber & \qquad -\frac{n_0\ell^2(q^2+1)}{p_0} +  24nq(2v_3(L)) +\varepsilon p \left( S\left(\frac{m}{n}\right) -\frac{m+\varepsilon}{n} \right)
\end{align*}
\end{proposition}
\begin{proof}
Since $M$ is the $\frac{p}{q}$-surgery of $K_y$,
\[ \frac{\lambda_{w}(M)}{2}= a_{2}(K_y)\frac{q}{p} - \frac{1}{24}S\left(\frac{q}{p}\right).\]

By the reciprocity law of the Dedekind symbol
\[ S\left(\frac{p}{q}\right)-\frac{p}{q} = -S\left(\frac{q}{p}\right)+\frac{q}{p}+\frac{1}{pq}-3.\]
The determinant $D$ of the linking matrix is given 
\[ D=\det \begin{pmatrix} \frac{m}{n} & \ell \\ \ell & \frac{p}{q} \end{pmatrix} = \frac{mp-nq\ell^2}{nq} = \varepsilon \frac{p}{nq}
\]
Therefore by Theorem \ref{theorem:Q-MP},
\begin{align*}
D \left( \frac{\lambda_w(S^{3}_L)}{2}-\frac{\varsigma}{8}\right ) &= 
\frac{\varepsilon p}{nq}\left(a_{2}(K_y)\frac{q}{p} - \frac{1}{24}S\left(\frac{q}{p}\right) -\frac{\varsigma}{8}\right)\\
&= a_2(K_x)\frac{p}{q}-\frac{p}{24q}-\frac{p}{24qn^2}+\frac{p\ell^2}{24q} \\
&+a_2(K_y)\frac{m}{n}-\frac{m}{24n}-\frac{m}{24nq^2}+\frac{m\ell^2}{24n}+2v_3(L)\\
& + \frac{\varepsilon p}{24nq}\left(S\left(\frac{m}{n}\right)-\frac{m}{n}-S\left(\frac{q}{p}\right)+\frac{q}{p}+\frac{1}{pq}-3\right) 
\end{align*}
By multiplying $24nq$ to both sides, we get
\begin{align*}
24\varepsilon a_2(K_y)q-\varepsilon pS\left(\frac{q}{p}\right)-3 \varepsilon \varsigma p &=24a_2(K_x)np -np- \frac{p}{n}+np\ell^2\\
& +24a_2(K_y)qm-qm-\frac{m}{q}+mq\ell^2+24nq(2v_3(L))\\
&+\varepsilon pS\left(\frac{m}{n}\right)-\frac{\varepsilon pm}{n}-\varepsilon pS\left(\frac{q}{p}\right)+\varepsilon q+\frac{\varepsilon}{q}-3\varepsilon p.
\end{align*}
Hence 
\begin{align*}
0 & = 24a_2(K_x)np+24a_2(K_y)q(m-\varepsilon)+(3\varepsilon \varsigma-3 \varepsilon -n)p+(np+mq)\ell^2\\
&\qquad +(\varepsilon-m)\frac{q^2+1}{q} +24 nq(2v_3(L)) +\varepsilon p \left( S\left(\frac{m}{n}\right) -\frac{m+\varepsilon}{n} \right).
\end{align*}
By Proposition \ref{proposition:1sthomology} (i), $\varepsilon-m=-\frac{n\ell^2q}{p} = -\frac{n_0 \ell^2 q}{p_0}$. Therefore we conclude
\begin{align*}
0 &= 24a_2(K_x)np+24a_2(K_y)q(m-\varepsilon)+(3\varepsilon \varsigma-3 \varepsilon -n)p+(np+mq)\ell^{2}\\
\nonumber & \qquad -\frac{n_0\ell^2(q^2+1)}{p_0} +  24nq(2v_3(L)) +\varepsilon p \left( S\left(\frac{m}{n}\right) -\frac{m+\varepsilon}{n} \right)
\end{align*}
\end{proof}

We note that the dedekind sum term gives one constraint for $S^{3}_L=M$ which is interesting in its own right.
\begin{proposition}
\label{proposition:Dsum}
Assume that $S^{3}_L=M$. Then
$p \left( S\left(\frac{m}{n}\right) -\frac{m+\varepsilon}{n} \right) \equiv 0 \pmod{p_0}$.
\end{proposition}
\begin{proof}
Since we have seen that $p_0 \mid \ell^2$,  $-\frac{n_0\ell^2(q^2+1)}{p_0} \in \Z$. Consequently, by Proposition \ref{proposition:CW}
\[ p \left( S\left(\frac{m}{n}\right) -\frac{m+\varepsilon}{n} \right) = \frac{p_0}{n_0}\left( nS\left(\frac{m}{n}\right) -(m+\varepsilon) \right) \in \Z. \]
Since $p_0$ and $n_0$ are coprime, this means that $\frac{1}{n_0}\left( nS\left(\frac{m}{n}\right) -(m+\varepsilon) \right) \in \Z$ hence
\[
p \left( S\left(\frac{m}{n}\right) -\frac{m+\varepsilon}{n} \right) = \frac{ p_0}{n_0}\left( nS\left(\frac{m}{n}\right) -(m+\varepsilon) \right) \equiv 0 \pmod{p_0}.
\]
\end{proof}

The following gives a further constraint for the factor $d_0=p_0\slash \gcd(p_0,\ell)$ in Proposition \ref{proposition:1sthomology} (ii).
 
\begin{proposition}
\label{prop:key}
Assume that $S^{3}_L=M$. Let $p_0=d_0^{2}p_0'$ where $d_0=p_0\slash \gcd(p_0,\ell)$. If $d_0\neq 1$ and $d_0 \mid 24$, $\gcd(d_0,p_0')\neq 1$.
\end{proposition}
\begin{proof}
By Proposition \ref{proposition:Dsum}, $p \left( S\left(\frac{m}{n}\right) -\frac{m+\varepsilon}{n} \right) \equiv 0 \pmod{d_0}$. Therefore when $d_0 \mid 24$, by Proposition \ref{proposition:CW} we get 
\[ \frac{n_0\ell^2(q^2+1)}{p_0} = n_0 p_0' \ell'^{2}(q^{2}+1) \equiv 0 \pmod{d_0}\]
Assume, to the contrary that $d_0$ and $p_0'$ are coprime. Since both $n_0,\ell'$ are coprime to $d_0$,
\[ (q^{2}+1) \equiv 0 \pmod{d_0}\]
When $d_0\neq 2$, this cannot happen. When $d_0=2$, 
\[ \frac{n_0\ell^2(q^2+1)}{p_0} = n_0 p_0' \ell'^{2}(q^{2}+1) \not \equiv 0 \pmod{4}.\]
On the other hand, by Proposition \ref{proposition:CW} 
\[ \frac{n_0\ell^2(q^2+1)}{p_0} = n_0 p_0' \ell'^{2}(q^{2}+1) \equiv 0 \pmod{4}\]
This is a contradiction.
\end{proof}

Theorem \ref{theorem:main} in introduction is a direct consequence of Proposition \ref{prop:key}.

\begin{proof}[Proof of Theorem \ref{theorem:main}]
Assume to the contrary that $S^{3}_L=M$.
By assumption $n=\Delta(s,\mu_K)=1$, $p=p_0$. Hence by Proposition \ref{proposition:1sthomology} (ii), $p=d_0^2p_0'$, $\ell=d_0p'_0\ell'$, where $d_0=p\slash \gcd(p,\ell)$.  Since $K$ is non-null-homologous, $p \nmid \ell$ so $d_0 > 1$. On the other hand, by the assumption of $p$, we get $d_0 \in \{2,3,6\}$, and $\gcd(d_0,p'_0)=1$. This contradicts with Proposition \ref{prop:key}.
\end{proof}

\section{Applications}

We prove various theorems stated in introduction.

\begin{proof}[Proof of Theorem \ref{theorem:main-finite}]
Let $M$ be a 3-manifold obtained by the $\frac{p}{q}$-surgery on a knot $K_y$ in $S^{3}$.

Let $K_x$ be a knot in $S^{3}$ corresponding to the knot $K$ in $M$, and let $L= K_x \cup K_y$. As we have seen, for the link $L=K_x \cup K_y$ and a slope $s$ of the knot $K$ in $M$, $M_K(s) = S^{3}_L(\frac{m}{n},\frac{p}{q})$ for some $m \in \Z$ and $n=\Delta(s,\mu_{K})$. 

We show that the Dehn surgery on a non-trivial knot $K$ in $M$ does not yield $M$, by showing $S^{3}_L \not \cong M$. Since we are assuming $M$ is an L-space, by Theorem \ref{theorem:L-space} in the following we assume that $K$ is not null-homologous. \\

(i): If $M$ is reducible, by \cite[Theorem 1.2]{gl2} $n=\Delta(s,\mu_{K})=1$ so $M_K(s) \not \cong M$ by Theorem \ref{theorem:main}. \\

In the case (ii), (iii), (iv) $M$ is small Seifert fibered so by Theorem \ref{theorem:Non-hyperbolic} it is sufficient to study the case $K$ is a hyperbolic knot in $M$.\\

(ii): By cyclic surgery theorem \cite{cgls}, the surgery slope $s$ satisfies $n=\Delta(s,\mu_{K})=1$ so $M_K(s) \not \cong M$ by Theorem \ref{theorem:main}. \\

(iii), (iv) : Assume to the contrary that $S^{3}_L(\frac{m}{n},\frac{p}{q})=M$. Let $c=\gcd(n,p)$, $p=cp_0$ and $n=cn_0$.

First we show that $(c,n)=(1,1),(2,2),(2,6),(3,3),(3,6)$. If $\pi_1(M)$ is finite, then by finite filling theorem \cite{bz}, $n=\Delta(s,\mu_K) \leq 3$. If $M$ is small Seifert fibered, then $n=\Delta(s,\mu_K) \leq 8$ by \cite{lame}. Since we assume that $p$ is square-free and coprime to $5$ and $7$, when $c\neq 1$, $(c,n)=(2,2),(2,6),(3,3),(3,6),(6,6)$. However, since we are assuming $6 \nmid p$, $(c,n) \neq (6,6)$.

If $c=\gcd(n,p)=1$, then the assertion follows from Proposition \ref{prop:key} hence we assume that $c=\gcd(n,p)>1$. 
Since we are assuming that $p$ is square-free, $d_0=p_0 \slash \gcd(p_0,\ell)=1$ and $p_0,\ell \not \equiv 0 \pmod{c}$. 

Assume that $(c,n)=(2,2)$ or $(c,n)=(2,6)$. Since $p_0,\ell,q \equiv 1 \pmod{2}$, $\frac{nq^{2}\ell^2}{p} \equiv 1 \pmod{2}$. On the other hand, 
by Proposition \ref{proposition:1sthomology} (i) $m=\frac{nq\ell^2}{p}\pm \varepsilon \equiv 0 \pmod{2}$ so $m \equiv 0 \pmod{2}$. Since $n$ and $m$ are coprime, this is a contradiction.

Next assume that $(c,n)=(3,3)$. Since $p_0,\ell \not \equiv 0 \pmod{3}$ and $p_0 \mid \ell$, we put $\ell=p_0 \ell'^{2}$, where $\ell' \not \equiv 0 \pmod{3}$. Then 
\begin{equation}\label{eqn:H1}
\frac{\ell^2(q^2+1)}{p_0} = p_0\ell'^{2}(q^{2}+1) \equiv -p_0 \pmod 3
\end{equation}
On the other hand, by Proposition \ref{proposition:CW}
\[  \frac{\ell^2(q^2+1)}{p_0} \equiv \varepsilon p_0\left( 3S\left(\frac{m}{3}\right) -(m+\varepsilon) \right) \pmod{3}\]
If $m\equiv 1 \pmod 3$, $3S\left(\frac{m}{3}\right) = 2$ so
\[\varepsilon p_0\left( 3S\left(\frac{m}{3}\right) -(m+\varepsilon) \right) \equiv  \begin{cases} 0 \pmod{3} & (\varepsilon=+1) \\ p_0 \pmod{3} &(\varepsilon=-1). \end{cases} \]
Similarly, if $m\equiv -1 \pmod 3$, $3S\left(\frac{m}{3}\right) = -2$ so
\[ \varepsilon p_0\left( 3S\left(\frac{m}{3}\right) -(m+\varepsilon) \right) \equiv \begin{cases} p_0 \pmod{3} & (\varepsilon=+1) \\
0 \pmod{3}& (\varepsilon=-1). \end{cases}\]
In any cases they contradict with (\ref{eqn:H1}).

The case $(c,n)=(3,6)$ is similar to the case $(c,n)=(3,3)$.
\end{proof}

\begin{proof}[Proof of Theorem \ref{theorem:cosmetic}]
Assume that $K$ admits diagram $D$ such that a crossing change at a crossing $c$ of $D$ preserves the knot type. 
Let $\gamma$ be a crossing arc of $c$, the segment $c \times [a,b] \subset \R^2\times \R$ connecting $p_D^{-1}(c)$, where $p_D:\R^3=\R^{2}\times \R \to \R^2$ is the diagram projection (see Figure \ref{fig:nugatory} (ii)).

Then $\gamma$ lifts to a knot $\widetilde{\gamma}$ in the double branched covering $\Sigma(K)$.
Let $K'$ be the knot obtained by the crossing change at $c$. Then by Montesinos trick $\Sigma(K')$ is obtained from $\Sigma(K)$ by a Dehn surgery on $\widetilde{\gamma}$ with slope $s$, such that $\Delta(\mu_{\widetilde{\gamma}},s)=2$, where $\mu_{\widetilde{\gamma}}$ denotes the meridian of $\widetilde{\gamma}$.

Since we assume that crossing change preserves the knot, $\Sigma(K') = \Sigma(K)$. Since $p$ is odd, $\gcd(\Delta(\mu_{\widetilde{\gamma}},s),p)=\gcd(2,p)=1$.
If $\widetilde{\gamma}$ is not null-homologous, by Theorem \ref{theorem:main}, for such a slope $s$, the Dehn surgery on $\widetilde{\gamma}$ does not produce $\Sigma(K)$. 

Therefore  $\widetilde{\gamma}$ is null-homologous. By  Theorem \ref{theorem:L-space}, $\widetilde{\gamma}$ is the unknot in $\Sigma(K)$. This implies that the crossing $c$ is nugatory by \cite[Proposition 3.3]{lm}.

\end{proof}

\begin{proof}[Proof of Corollary] \ref{cor:ten}
The knots $10_{65}, 10_{67}, 10_{77}, 10_{108}$ and $10_{164}$ satisfy the assumptions (a), (c). The unknotting number of $10_{164}$ is one. The $H(2)$-unknotting number of $10_{67}$ and $10_{108}$ are one \cite{nak,zek} so they satisfy the assumption (b)$'$.
The knots $10_{65}$ and $10_{77}$ satisfy the assumption (b); they are the Montesions knots $M(\frac{3}{4},\frac{1}{3},\frac{2}{3})$, $M(\frac{1}{2},\frac{2}{3},\frac{5}{2})$ and their double branched coverings are the $\frac{63}{5}$-surgery on the $(3,4)$-torus knot $T_{3,4}$ and the $\frac{63}{10}$-surgery on the $(2,3)$-torus knot $T_{2,3}$, respectively.
\end{proof}

\end{document}